\documentclass[11pt,a4paper]{article}

\usepackage{amsmath}
\usepackage{amssymb}
\usepackage{amsthm}
\usepackage{latexsym}
\usepackage{pstricks,hyperref}
\usepackage{amsbsy}
\usepackage{amscd}
\usepackage{mathrsfs}
\usepackage{txfonts}
\usepackage{amsfonts}
\usepackage{graphicx}
\newtheorem{thm}{Theorem}[section]
\newtheorem{cor}[thm]{Corollary}
\newtheorem{lem}[thm]{Lemma}
\newtheorem{exm}[thm]{Example}
\newtheorem{prop}[thm]{Proposition}
\newtheorem{defn}[thm]{Definition}
\newtheorem{rem}[thm]{Remark}
\usepackage{pstricks,color}

\begin{document}

\begin{center}
{\Large \bf Cotorsion pairs in the cluster category of a marked surface\ \footnote{Supported by the NSF of China
(Grants 11131101) }}

\bigskip

{\large Jie Zhang, Yu Zhou and
 Bin Zhu}
\bigskip

{\small
\begin{tabular}{cc}
Department of Mathematical Sciences & Department of Mathematical
Sciences
\\
Tsinghua University & Tsinghua University
\\
  100084 Beijing, P. R. China &   100084 Beijing, P. R. China
\\
{\footnotesize E-mail: jie.zhang@math.tsinghua.edu.cn} &
{\footnotesize E-mail: yu-zhou06@mails.tsinghua.edu.cn}
\end{tabular}
\begin{tabular}{c}
Department of Mathematical Sciences
\\
 Tsinghua University
\\
  100084 Beijing, P. R. China
\\
{\footnotesize E-mail: bzhu@math.tsinghua.edu.cn}
\end{tabular}

 \vspace{0.4cm} {\it Dedicated to Professor Idun Reiten on the occasion of her seventieth birthday.}}
\bigskip


\end{center}

\def\s{\stackrel}
\def\Longrightarrow{{\longrightarrow}}
\def\A{\mathcal{A}}
\def\B{\mathcal{B}}
\def\C{\mathcal{C}}
\def\D{\mathcal{D}}
\def\T{\mathcal{T}}
\def\R{\mathcal{R}}
\def\S{\mathcal{S}}
\def\H{\mathcal{H}}
\def\U{\mathscr{U}}
\def\V{\mathscr{V}}
\def\W{\mathscr{W}}
\def\X{\mathscr{X}}
\def\Y{\mathscr{Y}}
\def\Z{\mathcal {Z}}
\def\I{\mathcal {I}}
\def\add{\mbox{add}}
\def\Aut{\mbox{Aut}}
\def\coker{\mbox{coker}}
\def\deg{\mbox{deg}}
\def\dim{\mbox{dim}}
\def\End{\mbox{End}}
\def\Ext{\mbox{Ext}}
\def\Hom{\mbox{Hom}}
\def\Gr{\mbox{Gr}}
\def\id{\mbox{id}}
\def\Im{\mbox{Im}}
\def\ind{\mbox{ind}}
\def\Int{\mbox{Int}}
\def\mod{\mbox{mod}}
\def\ggz{\Gamma}
\def\bz{\beta}
\def\az{\alpha}
\def\gz{\gamma}
\def\da{\delta}
\def\fs{{\mathfrak{S}}}
\def\ff{{\mathfrak{F}}}
\def\zz{\zeta}
\def\thz{\theta}
\def\ra{\rightarrow}

\def \text{\mbox}

\hyphenation{ap-pro-xi-ma-tion}

\begin{abstract}

We study extension spaces, cotorsion pairs and their mutations in
the cluster category of a marked surface without punctures. Under
the one-to-one correspondence between the curves, valued closed
curves in the marked surface and the indecomposable objects in the
associated cluster category, we prove that the dimension of
extension
 space of two indecomposable objects in the cluster categories equals to the intersection number of the corresponding curves.
 By using this result, we prove that there are no non-trivial $t-$structures in the cluster categories when the surface is
connected. Based on this result,
 we give a classification of cotorsion pairs in these categories. Moreover we define the notion of
 paintings of a marked surface without punctures
 and their rotations. They are a geometric model of cotorsion pairs and of their mutations respectively.

\end{abstract}

\medskip


\section{Introduction}

The notion of torsion pairs (or torsion theory) in abelian
categories was introduced by Dickson \cite{D66} (see \cite{ASS06} for
further details). It is important in algebra and geometry \cite{BRe07}
and plays an important role in representation theory of algebras, in
particular in tilting theory \cite{ASS06}. The triangulated version of
torsion pairs was introduced by Iyama and Yoshino \cite{IY08} in their
study of mutation of cluster tilting subcategories, see also
\cite{KR07,BRe07}. Cluster tilting objects (or subcategories) appear
naturally in the study on the categorification of cluster algebras
\cite{BMRRT06}. They have many nice algebraic
 and combinatorial properties which have been used in the
categorification of cluster algebras \cite{K12,K08}. In particular,
there is bijection between cluster tilting objects in a cluster
category of an acyclic quiver and clusters of the corresponding
cluster algebra \cite{Re10, K12}.

Cluster tilting subcategories in triangulated categories are the
torsion classes of certain torsion pairs. In general, a triangulated
category (even a $2-$Calabi-Yau triangulated category) may not admit
any cluster tilting subcategories \cite{KZ08,BIRS09}. However it always
admits torsion pairs, for example, the trivial torsion pair: (the
whole category, the zero category). Consequently triangulated
categories admit cotorsion pairs, since in a triangulated category
$\C$ with shift functor $[1]$, $(\X,\Y)$ is a cotorsion pair in $\C$
if and only if $(\X, \Y[1])$ is a torsion pair.

The geometric construction of cluster categories of type $A$ was
given by Caldero-Chapton-Schiffler in \cite{CCS06}, see also \cite{Sch08} for type
$D$, \cite{HJ12} for type $A_{\infty}$, and \cite{BM10} for (abelian) tube
categories. Many algebraic properties (e.g the extension dimensions,
Auslander-Reiten triangles) of these cluster categories were studies
in geometric terms.  Torsion pairs in the cluster categories of
$A_{\infty}$ were
 classified by Ng in terms of Ptolemy diagrams of a $\infty-$gon
\cite{Ng10}. By using the idea of Ng, Holm-J{\o}rgensen-Rubery
\cite{HJR11} gave a classification of torsion pairs in the cluster
categories of type $A_n$ by the Ptolemy diagrams of a regular
$(n+3)-$gon. We also note that Baur-Buan-Marsh \cite{BBM11} gave a
classification of torsion pairs in the (abelian) tube categories.
Recently Holm-J{\o}rgensen-Rubey announced a classification of torsion
pairs in cluster tubes.

 Let $(S,M)$ be a pair consisting of a compact oriented Riemann
surface $S$ with non-empty boundary and a finite set $M$ of marked
points on the boundary of $S$, with at least one marked point on
each component of the boundary.  We do not assume the surface to be connected,
but we assume that $S$ has
no component homeomorphic to a monogon, digon, or triangle. Let
$\C_{(S,M)}$ be the generalized cluster category in the sense of
Amiot \cite{Ami09} associated to $(S,M)$. It is a $2-$Calabi-Yau
triangulated category with cluster tilting objects. In \cite{BZ11} the
authors proved that there is a bijection between the indecomposable
objects in $\C_{(S,M)}$ and the curves, valued closed curves in
$(S,M)$, they also gave a geometric description of Auslander-Reiten
triangles. We note that the cluster algebras associated to a marked
surface (with or without punctures) has been studied by Fomin,
Shapiro and Thurston in \cite{FST08} and many others.

Let $\C$ be a $2-$CY triangulated category, $(\X,\Y)$ a cotorsion
pair in $\C$. Denote by $I$ the core of $(\X,\Y)$, i.e.
 the intersection of $\X$ and $\Y$. In \cite{ZZ11}, the authors
  defined the $\D-$mutation of $(\X,\Y)$ for $\D\subset I$,
   and proved that the $\D-$mutation of $(\X,\Y)$ is a cotorsion pair.
\\

    In this paper, we prove
that the dimension of extension spaces between two string objects in
$\C_{(S,M)}$ equals to the intersection number of corresponding
curves in $(S,M)$. By using this result, we show that there is no
non-trivial $t-$structures in the cluster category $\C_{(S,M)}$ when
$(S,M)$ is connected. We give a classification of cotorsion pairs in
the cluster category $\C_{(S,M)}$ by  using the terms of curves and
valued closed curves in $(S,M)$.  Furthermore, we define the notion
of paintings of $(S,M)$ which serve as a geometric model of
cotorsion pairs in $\C_{(S,M)}$. We also define the rotation of
paintings, which is proved to be compatible to the mutation of
cotorsion pairs.
\\

This paper is organized as follows: In Section 2, some basic
  definitions and results on cotorsion pairs are recalled in the first subsection. In the second subsection,
   we recall the
  Br\"ustle-Zhang's bijection $X_{?}^{(S,M)}$ from the curves and valued
  closed curves in a marked surface $(S,M)$ to the indecomposable objects in the associated cluster category
  $\C_{(S,M)}$, and the geometric construction of Auslander-Reiten triangles in \cite{BZ11}.
  In the final subsection, we recall the Iyama-Yoshino reduction of
  $2-$Calabi-Yau categories and its geometric interpretation given
  by Marsh-Palu for the cluster categories $\C_{(S,M)}$ in \cite{MP11}.
  We prove that the extension space of two objects in the
subfactor categories is isomorphic to the extension space of the
  corresponding objects in the original categories, which will be used latter.
 In Section 3, by using Crawley-Boevey's description of the basis of Hom-space of two string
 modules \cite{C-B89}, we prove the main result (Theorem 3.5) in this section:
 the dimension of extension space between two string objects in $\C_{(S,M)}$
 is the same with intersection number of the corresponding curves in
 $(S,M)$.
   In Section 4, the first main result is that when $S$ is connected, the cluster category $\C_{(S,M)}$ has no non-trivial
   $t-$structure (Theorem 4.3). This allows us to give a classification of cotorsion
   pairs in $\C_{(S,M)}$ (Theorem 4.5), which is the second main result in this section.
We define, in Section 5, the notion of painting of $(S,M)$ and its
 rotation. Under a correspondence between cotorsion pairs and paintings of
 $(S,M)$, we give a geometric description of the mutation of cotorsion
 pairs by rotation of paintings.

Throughout this paper, $k$ denotes an algebraic closed field.

\section{Preliminaries}

\subsection{Triangulated categories and the abelian quotients}

In this section we recall some basic notations and facts on cotorsion pairs in a
triangulated category $\C$ with shift functor $[1]$. By $X\in \C$, we mean that $X$ is an object
of $\C$. For a subcategory $\X$ of $\C$, denoted $\X\subset\C$, we always assume that $\X$
is a full subcategory and closed under taking isomorphisms, direct
summands and finite direct sums. Moreover, let
$$\X ^{\bot}=\{Y\in\C\mid\Hom_{\C}(X,Y)=0 \mbox{~for any~} X\in \X\}$$
and
$$^{\bot}\X=\{Y\in\C\mid\Hom_{\C}(Y,X)=0 \mbox{~for any~} X\in \X\}.$$

For two subcategories $\X,
\Y$, by $\Hom_{\C}(\X, \Y)=0$, we mean that $\Hom_{\C}(X,Y)=0$
for any $X\in \X$ and any $Y\in \Y$. A subcategory $\X$ of $\C$ is said to be a rigid subcategory
if $\Hom(\X,\X[1])=0$. We denote by $\Ext_{\C}^n(X,Y)$ the space
$\Hom_{\C}(X,Y[n])$. Let
$$\X*\Y=\{Z\in\C\mid ~~\text{there exists a triangle }  X\rightarrow Z\rightarrow Y\rightarrow X[1]
 \mbox{~in $\C$~with~} X\in \X, Y\in \Y\}.$$
It is  easy to see that $\X*\Y$ is
closed under taking isomorphisms and finite direct sums. A subcategory $\X$ is said to be closed
under extensions (or an extension-closed
 subcategory) if $\X*\X\subset \X$. Note that $\X*\Y$ is closed under taking direct
summands by Proposition 2.1(1) in \cite{IY08} if $\Hom(\X,\Y)=0$. Therefore, in this case $\X*\Y$ can
be understood as a subcategory of $\C$.
 We recall the definition of cotorsion pairs in a triangulated category $\C$ from \cite{Na11} \cite{IY08}.

\begin{defn}\label{def} Let $\X$ and $\Y$ be subcategories of a triangulated category $\C$.
\begin{itemize}
\item[$1.$] The pair $(\X,\Y)$ is a cotorsion pair if
$$\Ext_\C^1(\X,\Y)=0\text{ and }\C=\X*\Y[1]\text{.}$$

\item[$2.$] A $t-$structure $(\X, \Y )$ in $\C $ is a cotorsion pair such that $\X$ is closed under $[1]$ (equivalently
$\Y $ is closed under $[-1]$). In this case $\X\bigcap \Y[2]$ is an
abelian category, which is called the heart of $(\X,\Y)$
\cite{BBD81} \cite{BRe07}.

\item[$3.$] The subcategory  $\X$ is said to be a cluster tilting
 subcategory if $(\X, \X)$ is a cotorsion pair \cite{BMRRT06}. We say that a basic object $T$ is a cluster tilting object
 if its additive closure $\add T$ is a cluster tilting subcategory.
\end{itemize}
\end{defn}
Moreover, we call the subcategory $I=\X\bigcap \Y $ the core of the cotorsion pair $(\X,\Y)$.

 \begin{rem} A pair
$(\X,\Y)$ of subcategories of $\C$ is said to be
 a torsion pair if $\Hom(\X,\Y)=0$ and $\C=\X*\Y$. In this case,
 $I=\X \bigcap\Y[-1]$ is called the core of the torsion pair. Moreover,
 a pair $(\X,\Y)$ is a cotorsion pair if and only if $(\X, \Y [1])$
is a torsion pair in the sense of \cite{IY08}.
\end{rem}

Recall that a subcategory $\X$ is said to be contravariantly finite in $\C$,
 if any object $M\in \C$ admits a right
$\X-$approximation $f:X\rightarrow M$, which means that any map from
$X'\in \X$ to $M$ factors through $f$. The left $\X-$approximation
of $M$ and covariantly finiteness of $\X$ can be defined dually.
$\X$ is called functorially finite in $\C$ if $\X$ is both
covariantly finite and contravariantly finite in $\C$.
Note that if $(\X,\Y)$ is a torsion pair, then $\X={}^{\bot}\Y$, $\Y=\X^\bot$, and
it follows that $\X$ (or $\Y$) is a contravariantly (covariantly, respectively) finite and extension-closed subcategory of
$\C$.

\medskip

For subcategories
$\D\subset\X$ of $\C$, the quotient category $\X/\D$ of $\X$ by $\D$
has the same objects as $\X$, and its morphism spaces are defined by
$$\Hom_{\X/\D}(X,Y)=\Hom_\C(X,Y)/\D(X,Y),$$ where $\D(X,Y)$
 is the subset of $\Hom_\C(X,Y)$ consisting of morphisms which factor through some object in $\D$.
 In particular, we have the quotient category $\C/\D$ which is an
additive category.

\medskip

\subsection{The cluster category of a marked surface}\label{marked surface}

\medskip

Let $(S,M)$ be a marked surface without punctures, i.e. $S$ is a
compact oriented Riemann surface with $\partial S\neq\emptyset$ and
$M$ is a finite set of marked points on the boundary of $S$ such
that there is at least one marked point on each connected component
of boundary of $S$. Note that the cluster algebras associated to
$(S,M)$ has been studied by many papers, see for example, \cite{FST08}. We
also note that we do not assume the surface to be connected.

\medskip

By a curve $\gz$ in $(S,M)$, we mean the image of a continuous function $\gz:[0,1]\rightarrow S$ such that $\gz(0), \gz(1)\in M$ and $\gz(t)\notin M$ for $0<t<1$, which is not homotopic to a boundary segment. By a closed curve $b$ in $(S,M),$ we mean the image of a non-contractible continuous function $b: \mathbb{S}^1 \rightarrow S\setminus\partial S$ where $\mathbb{S}^1$ denotes the unit circle in the complex plane, a pair $(\lambda, b)$ with $\lambda\in k^*$ is said to be a valued closed curve.

We say two curves $\gz$ and $\da$ are homotopic if they are homotopic relative to $M$.
 Two valued closed curves $(b_0,\lambda_0)$ and $(b_1,\lambda_1)$ are called homotopic if $b_0$ and $b_1$ are
 homotopic and $\lambda_0=\lambda_1$. Each curve $\gamma$ or valued closed curve $(b,\lambda)$ is considered
 up to homotopy. For two curves $\gz$ and $\da$ in $(S,M)$, we denote
 by $\Int(\gz,\da)$ the minimal intersection number of two representatives of the homotopy
 classes of $\gz$ and $\da$ (the intersection at the endpoints do not count).
 An arc is a curve $\gz$ in $(S,M)$ such that $\Int(\gz,\gz)=0$. Two arcs $\gz$, $\da$ are called
 compatible if $\Int(\gz,\da)=0$. A triangulation $\ggz$ of $(S,M)$ is any maximal collection of compatible arcs.

\medskip

Recall that each triangulation
$\Gamma$ yields a quiver with potential $(Q_\ggz, W_{\Gamma})$:
\begin{itemize}
\item[(1)] $Q_{\ggz}=(Q_0,Q_1)$ where the set of vertices $Q_0$ are indexed
by the arcs of $\Gamma$. Whenever there is a triangle $\Delta$ having $i$ and $j$ in $\ggz$ as edges, with $j$ following
$i$ in the clockwise orientation (which is induced by the orientation of
$S$), then there is an arrow from $i$ to $j$.

\item[(2)]Each internal triangle $\Delta$ whose edges
are all in $\ggz$ yields a unique 3-cycle
$\alpha_{\Delta}\beta_{\Delta}\gamma_{\Delta}$ (up to cyclic permutation), the potential $W_{\Gamma}$ is then defined as
the sum of all 3-cycles arising from internal triangles:
$$W_\ggz =\displaystyle\sum_\vartriangle\alpha_\vartriangle \beta_\vartriangle\gz_\vartriangle.$$

\end{itemize}

It is proved in \cite{L-F08, ABCP10} that the Jacobian algebra
$J(Q_{\Gamma},W_{\Gamma})$ is a finite dimensional string algebra
(see more details in \cite{BRi87} or in section \ref{string-alg} for string algebras).

\medskip

The cluster category of a marked surface is defined by \cite{Ami09}. In fact,
since the Jacobian algebra $J(Q_{\Gamma},W_{\Gamma})$ is
finite-dimensional, then there is a generalized cluster
category $C_{(Q_\ggz,W_\Gamma)}$ associated to $(Q^{op}_\ggz,W^{op}_\Gamma)$ which is
2-CY, Hom-finite and admits a cluster-tilting object $T_{\Gamma}$
such that

$$C_{(Q_\ggz,W_\Gamma)}/T_\Gamma \overset{\sim}{\longrightarrow}\mod J(Q_\ggz,W_\ggz)$$

under the functor $\Hom_{\C_{(Q_\ggz,W_\ggz)}}(T_\Gamma[-1],-)$ (see
\cite{KR07}, \cite{KZ08}). This cluster category is in fact independent of the
choice of the triangulation $\Gamma$ \cite{KY11}, since any two
triangulations are related by flips which correspond to mutations of
the corresponding quivers with potential. We can denote it by
$\C_{(S,M)}$.

\medskip

Recall that by using the above equivalence of two categories, the indecomposable objects in $\C_{(S,M)}$ are
indexed by curves and valued closed curves in $(S,M)$ which are called string objects and band objects respectively \cite{BZ11}.

We denote the indecomposable object in $\C_{(S,M)}$ corresponding to
a curve $\gamma$ or a valued closed curve $(b,\lambda)$ by
$X_{\gamma}^{(S,M)}$ or $X_{(\lambda, b)}^{(S,M)}$, respectively.
When no confusion can arise, we omit the superscript $(S,M)$. Since
any subcategories which we consider are closed under taking
isomorphisms, finite direct sums and direct summands, so each of
them is determined by its indecomposable objects. Therefore there is
a bijection $V\mapsto \X_V$ from the collections $V$ of curves and
valued closed curves in $(S,M)$ to the subcategories $\X_V$ of
$\C_{(S,M)}.$ Under this bijection, the collections $I$ of compatible arcs correspond
to rigid subcategories $\X_I$.

\medskip

The irreducible morphisms in $\C_{(S,M)}$ are also described in \cite{BZ11} by elementary pivot moves:
For a curve $\gamma$ in $(S,M)$, we denote the endpoints of $\gamma$ by $s(\gamma)$ and $e(\gamma)$ respectively.
Note that the orientation of $S$ induces an orientation on each boundary component of $S$.
The curve which is obtained from $\gamma$ by moving $s(\gamma)$ anticlockwise (resp. $e(\gamma)$) to the next $k-$th point is denoted by
by ${}_{s^k}\gamma$ (resp. $\gamma_{e^k}$) on the same boundary (see the following picture).
\begin{center}
\includegraphics[height=1.4in]{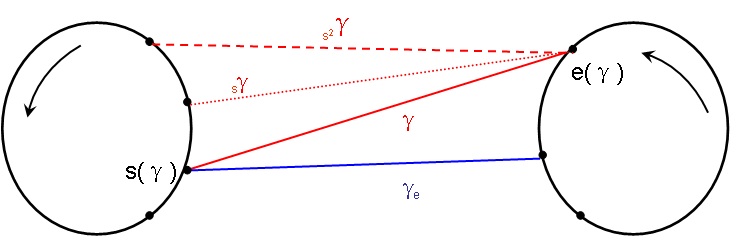}
\end{center}
We say $\gamma_{e}$ or ${}_{s}\gamma$ is obtained from $\gz$ by elementary pivot moves.
We summarize some more results in \cite{BZ11} which will be mentioned in following sections:
\begin{itemize}
\item The shift functor in $\C$ can be described as $X_{\gamma}[1]=X_{{}_{s}\gamma_{e}},$ then it makes sense that we denote
${}_{s^k}\gamma_{e^k}$ by $\gamma[k].$ Moreover,
the AR-triangles
between string objects in $\C_{(S,M)}$ can be described as
$$X_{{}_{s}\gamma_{e}}\rightarrow X_{{}_{s}\gamma}\oplus X_{\gamma_{e}}\rightarrow X_{\gamma}\rightarrow X_{{}_{s}\gamma_{e}}[1]$$
and band objects are stable under shift functor, that is $X_{(\lambda,b)}=X_{(\lambda,b)}[1].$

\item For a curve $\gamma$ in $(S,M)$, $\gamma$ has no self-intersections if and only
if $\Ext^1_{\C_{(S,M)}}(X_\gamma,X_\gamma)=0$ [Theorem 5.1, \cite{BZ11}. If we take two different curves
 $\gamma$ and $\delta$, then $\Int(\gamma,\delta)\neq0$ implies $\Ext^1_{\C_{(S,M)}}(X_\gamma,X_\delta)\neq 0\neq \Ext_{\C_{(S,M)}}^1(X_\delta,X_\gamma)$.

\end{itemize}

\subsection{Compatibility between IY subfactor triangulated categories and Cutting along arcs}

Let $\C$ be a $2-$CY triangulated $k-$category and $\D$ a
functorially finite rigid subcategory of $\C$. Denote
 $\Z={}^\bot\D[1]=\D[-1]^\bot$. Iyama and Yoshino \cite{IY08} proved
that the quotient category $\Z/\D$ is a triangulated category. The
shift functor $\langle1\rangle$ is induced by the triangle
$X\s{f}\rightarrow D_X\s{g}\rightarrow X\langle1\rangle\rightarrow
X[1]$ in $\C$ for $X\in\C$, where the morphism $f$ is a left
$\D-$approximation of $X$ and $g$ is a right $\D-$approximation of
$X\langle1\rangle$. This triangulated category, denoted by
$\C_{\D}$, is called subfactor triangulated category of $\C$.

The following lemma used in the next section follows straightly from the structure of subfactor category.
We give a proof here for the convenience of the reader.

\begin{lem}\label{Ext=}
Keep notations above. For any two objects $X,Y\in\Z$,
$$\Ext_\C^1(Y,X)\cong\Ext_{\C_\D}^1(Y,X)$$
as $k-$vector spaces.
\end{lem}

\begin{proof}
Applying $\Hom_{\C}(Y,-)$ to the triangle $X\s{f}\rightarrow D_X\s{g}\rightarrow X\langle1\rangle\rightarrow X[1]$,
we get an exact sequence
$$\Hom_{\C}(Y,X)\s{\Hom_{\C}(Y,f)}\rightarrow\Hom_{\C}(Y,D_X)
\s{\Hom_{\C}(Y,g)}\rightarrow\Hom_{\C}(Y,X\langle1\rangle)\rightarrow\Hom_{\C}(Y,X[1])\rightarrow
0$$
where $0=\Hom_\C(Y,D_X[1])$ by $D_X\in\D$ and $Y\in\Z$.
Therefore $$\Hom_{\C}(Y,X[1])\cong\Hom_{\C}(Y,X\langle1\rangle)/\Im(\Hom_{\C}(Y,g))$$ as $k-$vector spaces. Since $g$ is a right $\D-$approximation of $X\langle1\rangle$, every morphism from $Y$ to $X\langle1\rangle$ factoring through $\D$ factors through $g$. Then $\Im(\Hom_{\C}(Y,g))=\D(Y,X\langle1\rangle)$. So
$$\Hom_{\C}(Y,X[1])\cong\Hom_{\C_\D}(Y,X\langle1\rangle)$$
which means that $\Ext_\C^1(Y,X)\cong\Ext_{\C_\D}^1(Y,X)$.
\end{proof}

\medskip

Let $(S,M)$ be a marked surface without punctures and $\C=\C_{(S,M)}$ be
the associated cluster category. Given a rigid subcategory $\X_I$ of
$\C_{(S,M)}$ where $I$ is a
collection of some compatible arcs in $(S,M).$ We denote by $(S,M)/I$ the new marked surface obtained from $(S,M)$
by cutting successively along each arc in $I$ and then removing
components which are homeomorphic to a triangle. We denote by $\V(S,M)$ the collection of all curves and valued closed curves in
$(S,M)$. By $\V_{I}(S,M)$, we mean the collection of all curves and closed curves in $(S,M)$ which do not belong to $I$ such that
they do not cross any arcs in $I$.
\begin{prop}[\cite{MP11}]\label{propMP} There is an equivalence of categories
$$\pi_{I}: \C_{\X_I}\overset{~~~\thicksim~~~}{\longrightarrow} \C_{(S,M)/{I}}$$
such that
$$\pi_I(X^{(S,M)}_\gamma)= X^{(S,M)/I}_{\gamma} \mbox{~~and~~} \pi_I(X^{(S,M)}_{(b,\lambda)})= X^{(S,M)/I}_{(b,\lambda)}$$
for any $\gz, (b,\lambda)\in\V_{I}(S,M)$.
\end{prop}

\section{Intersection number and Extension dimension}
We consider in this section curves in $(S,M)$ or string objects in $\C_{(S,M)}$, and reveal a link
between the intersection number of two curves and extension dimension of two corresponding string objects in $\C_{(S,M)}$.
We first recall a construction of a basis of the Hom-space between string modules \cite{S99,C-B89}.
\subsection{Maps between string modules}\label{string-alg}
Recall from \cite{BRi87} that a finite-dimensional algebra $A$ is a
string algebra if there is a quiver $Q$ and an admissible ideal $I$
such that  $A=kQ/I$ and the following conditions hold:
\begin{itemize}
\item[(S1)]At each vertex of $Q$ start at most two arrows and stop at most two arrows.
\item[(S2)]For each arrow $\alpha$ there is at most one arrow $\beta$ and at most one arrow
$\da$ such that $\alpha\beta \not\in I $ and $\da\alpha\not\in I.$
\end{itemize}

\medskip

For each arrow $\bz$, $s(\bz)$ (resp. $e(\bz)$) denotes its starting point (resp. its ending
point). We denote by $\bz^{-1}$ the formal inverse of $\bz$ with $s(\bz^{-1})=e(\bz)$ and $e(\bz^{-1})=s(\bz)$.
A word $w=\az_n\az_{n-1}\cdots\az_1$ of arrows and their formal inverses is called a string if $\az_{i+1}\neq \az_i^{-1}, e(\az_i)=s(\az_{i+1})$ for all $1\leq i\leq n-1$, and no subword nor its inverse is in $I$.
 Hence, a string can be viewed as a walk in $Q:$
$$w: x_1 \frac{~~~\az_1~}{}x_2 \frac{~~~\az_2~}{}\cdots x_{n-1} \frac{\az_{n-1}}{}x_{n}\frac{~~~\az_n~}{}x_{n+1}$$
where $x_i$ are vertices of $Q$ and $\alpha_i$ are arrows in either direction. We denote its length by $l(w)=n$.

A band $b=\az_n\az_{n-1}\cdots\az_2\az_1$ is defined to be a
string $b$ with $e(\az_1)=s(\az_n)$ such that each power $b^m$ is a string, but $b$ itself is not a proper
power of any string.

Recall in \cite{BRi87} that each string $w$ defines a unique string
module $M(w)$, each band $b$ yields a family of band modules
$M(b,n,\phi)$ with $n\geq 1$ and $\phi\in \mbox{Aut}(k^n).$  We
refer \cite{BRi87} for more definitions on sting modules and band
modules.

\medskip

A string $w=\az_1\az_2\cdots\az_n$ with all $\az_i\in Q_1$ is called direct string, and
a string of the form $w^{-1}$ where $w$ is a direct string is called inverse string. We denote
by $\mathcal{S}$ the set of all strings. For each arrow $\az \in Q_1$, let $U_\az$ and $V_\az$
 be inverse strings (as long as possible) such that $N_\az= U_\az \az V_\az$ is a string, see the
 following figure:
\begin{center}
\includegraphics[height=.3in]{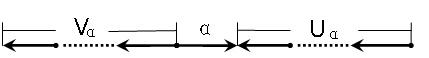}
\end{center}

For a string $w\in\mathcal{S},$ define
${\mathcal{S}}_w=\{(E,w',F)\mid E,w',F\in{\mathcal{S}}, w=Ew'F\},$
we call $(E,w',F)$ a factor string of $w$ if
 \begin{itemize}
  \item $l(E)=0$ or $E=\az_1\cdots\az_n$ with $\az_n\in Q_1,$
  \item $l(F)=0$ or $F=\bz_1\cdots\bz_m$ with $\bz_1^{-1}\in Q_1.$
\end{itemize}
Dually, $(E,w',F)$ is said to be a substring of $w$ if the following hold:
 \begin{itemize}
 \item $l(E)=0$ or $E=\az_1\cdots\az_n$ with $\az^{-1}_n\in Q_1,$
  \item $l(F)=0$ or $F=\bz_1\cdots\bz_m$ with $\bz_1\in Q_1.$
\end{itemize}
Denote by $\ff(w)$ and $\fs(w)$ the set of all factor strings and substrings of $w$ respectively, and we define $\ff(0)=\fs(0)=\emptyset$ for a zero module.
Let $w$ and $v$ be two strings, a pair $(E_1,w_0,F_1)\times (E_2, v_0, F_2)\in \ff(w)\times\fs(v)$ is said to be an admissible pair
if $w_0=v_0$ or $w_0^{-1}=v_0.$ It is easy to understand the admissible pair if one has the following picture in mind.
\begin{center}
\includegraphics[height=1.6in]{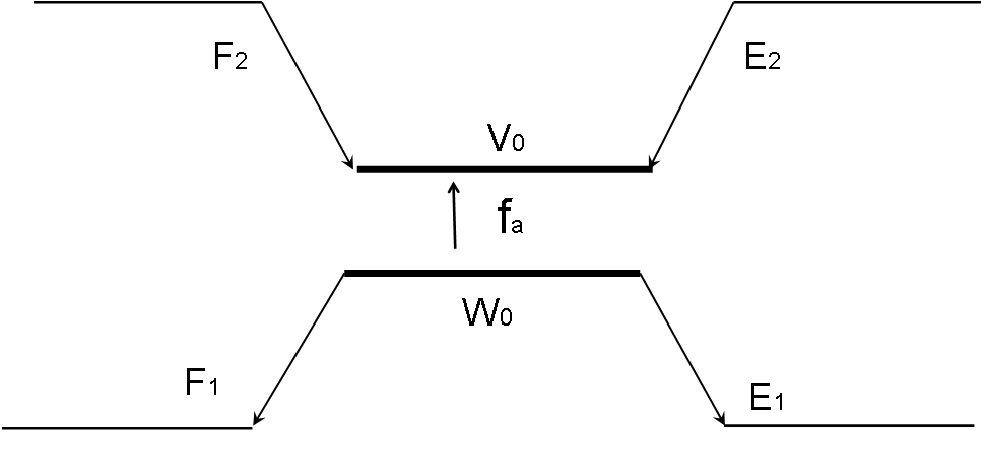}
\end{center}
Recall that each admissible pair $a\in \ff(w)\times\fs(v)$ as above yields a canonical module homomorphism
$f_a: M(w)\longrightarrow M(v)$ by identifying factor module $M(w_0)$ of $M(w)$ given by factor string $w_0$ to
submodule $M(v_0)$ of $M(v)$ related to substring $v_0$ of $v$. A basis of Hom-space of two string modules can be described as follows:
\begin{thm}[\cite{C-B89}]\label{C-B} Consider two strings $w$ and $v$. Then $\{f_a\mid a\in \ff(w)\times\fs(v)\}$ is a basis of $\mbox{Hom}(M(w),M(v))$.
\end{thm}

\subsection{Relation between intersection number and extension dimension}

Let $\ggz$ be a triangulation of $(S,M)$ and $J(Q_\ggz,W_\ggz)$ be the associated string algebra defined in section \ref{marked surface}. For each curve $\gamma$ with $d=\sum_{\delta\in\Gamma}\Int(\delta,\gamma),$ let
$\delta_1,\ldots,\delta_d$ be the arcs of $\Gamma$ that intersect
$\gamma$ in a fixed orientation of $\gamma$. See the following picture

\begin{center}
\includegraphics[height=1.6in]{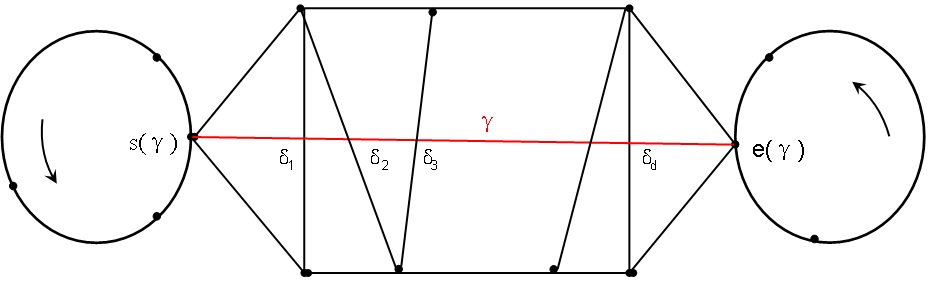}
\end{center}

Then we obtain a
string $w(\Gamma,\gamma)$ as follows
$$\da_1\frac{~~~~~~}{}\da_2 \frac{~~~~~~}{}\cdots \frac{~~~~~~}{} \da_d$$
in $J(Q_\Gamma,W_\Gamma)$.
We denote by $M(\Gamma,\gamma)$ the corresponding string module. Recall that the map $\gamma\rightarrow M(\Gamma,\gamma)$ gives a bijection between the homotopy classes of curves in $(S,M)$ which are not in $\Gamma$ and the isoclasses
of string modules of $J(Q_{\Gamma},W_{\Gamma})$ \cite{ABCP10}. Analogously, there is a bijection between
powers $b^n$ of bands $b$ of $J(Q_{\Gamma},W_{\Gamma})$ and the homotopy classes of closed curves in $(S,M)$.

Let $M(\ggz,\gz)$, $M(\ggz,\da)$ be the corresponding string modules in $\mbox{mod}J(Q_\ggz,W_\ggz).$ We assume that $\gz$ intersects $\da$ at $A_1,\ldots, A_d$ with $d=\Int(\gz,\da).$
It is easy to imagine that most intersections yield a common subword $w$ for $w(\ggz,\gz)$ and $w(\ggz, \da)$ as follows:
\begin{center}
\includegraphics[height=1.5in]{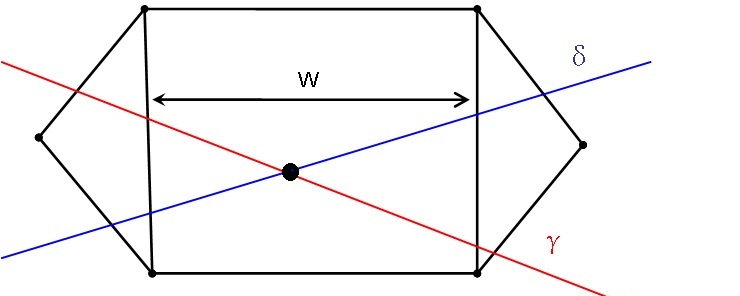}
\end{center}
However, by definition of triangulation, it is not true when the
segment between $A_i$ to the endpoints of $\gz$ or $\da$ does not
cross any  arcs of $\ggz$. See for example,
\begin{center}
\includegraphics[height=2in]{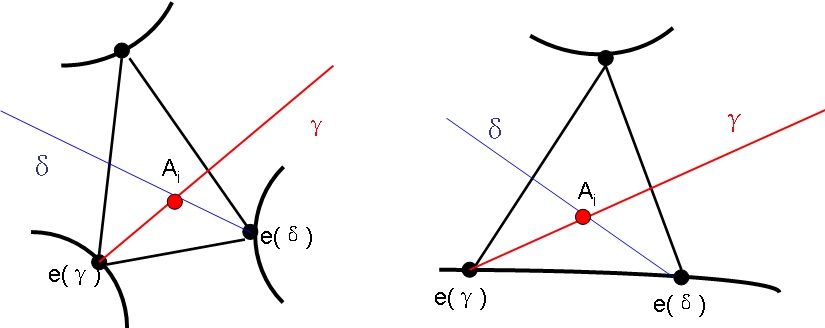}
\end{center}
To avoid this, we add two more marked points $p^1$, $p^2$ lying between $a$ on the same boundary
for each endpoint $p\in\{e(\gz),s(\gz),e(\da),s(\da)\}$.
 Therefore, we get a new set of marked points $M'$, then we form a new triangulation $\ggz_1=\ggz\cup\ggz_0$ where $\ggz_0$
 contains arcs in $(S, M')$ which are homotopic to boundary segments in $(S,M)$. See the following picture for example:
\begin{center}
\includegraphics[height=2.5in]{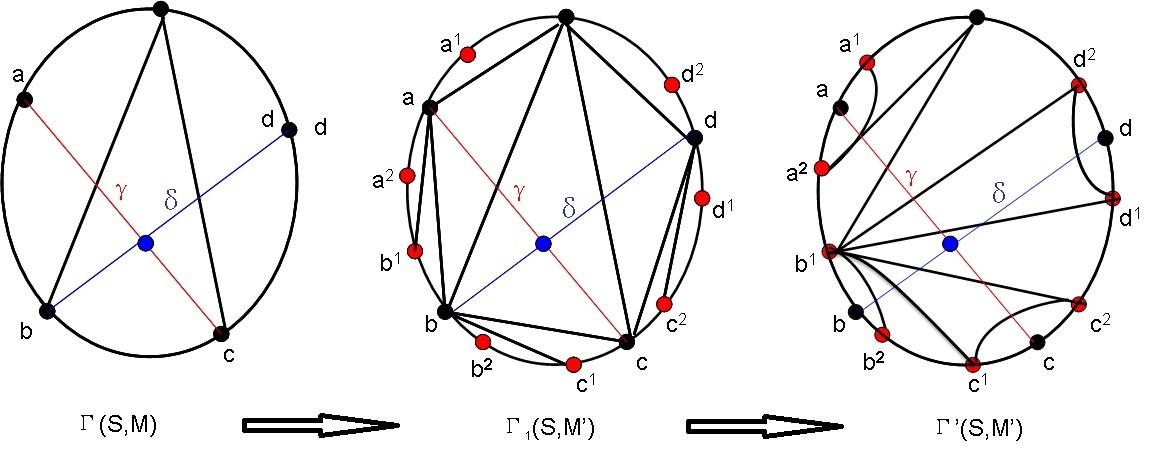}
\end{center}

Let $\C_{(S,M')}$ be the cluster category corresponding to $(S,M')$,
$X'_\gz$ and $X'_\da$ be the string objects in $\C_{(S,M')}$ corresponding to
$\gz$ and $\da$. Then by Proposition \ref{propMP}, we have
$$\C_{(S,M)}\cong\C_{(S,M')/\ggz_0}.$$
Then Lemma \ref{Ext=} implies
$$\mbox{Ext}^1_{\C_{(S,M)}}(X_\gz,X_\da)\backsimeq\mbox{Ext}^1_{\C_{(S,M')}}(X'_\gz,X'_\da).$$
Therefore, we can study $\mbox{Ext}^1_{\C_{(S,M')}}(X'_\gz,X'_\da)$ instead of $\mbox{Ext}^1_{\C_{(S,M)}}(X_\gz,X_\da)$.
For each endpoint $p\in\{e(\gz),s(\gz),e(\da),s(\da)\}$, we take an arc $p^1p^2$ which forms a triangle with boundary arc $pp^1$ and $pp^2$ (see the above picture). By adding more arcs, we get a new triangulation $\ggz'$ of $(S,M')$ where each intersection of $\gz$ and $\da$ induces a common subword for $w(\ggz', \gz)$ and $w(\ggz',\da)$ in $\mbox{mod}J(Q_{\ggz'},W_{\ggz'})$. Moreover,
$$\mbox{Ext}^1_{\C_{(S,M)}}(X_\gz,X_\da)\backsimeq\mbox{Ext}^1_{\C_{(S,M')}}(X'_\gz,X'_\da)\backsimeq\mbox{Ext}^1_{\C_{(Q_{\ggz'},W_{\ggz'})}}(X'_\gz,X'_\da).$$

To compare the intersection number of $\gz$ and $\da$ with the
dimension of $\mbox{Ext}^1_{\C_{(S,M)}}(X_\gz,X_\da)$, we can study
$\C_{(S,M')}$ instead of studying $\C_{(S,M)}.$

Note that curves in $(S,M)$ can be viewed as curves in $(S,M'),$ and
their intersection numbers do not change. For convenience, we denote
$\ff(w(\ggz',\gz))$ (resp. $\fs(w(\ggz',\gz))$) by $\ff'(\gz)$
(resp. $\fs'(\gz)$) for each curve $\gz$ in $(S,M).$

\begin{lem}\label{admissible pair}Let $\gz$ and $\da$ be two curves in $(S,M)$, then each intersection of $\gz$ and $\da$ induces
an admissible pair either in $\ff'(\gz)\times \fs'(_s\da_e)$ or in $\ff'(\da)\times\fs'(_s\gz_e).$
\end{lem}
\begin{proof} We fix orientations of $\gz$ and $\da,$ and take one intersection $A$ of them.
Note that $\gz$ and $\da$ play a same role to each other, we only prove the case when $\gz$ and $\da$ can be described locally as follows (related to $A$):
\begin{center}
\includegraphics[height=1.9in]{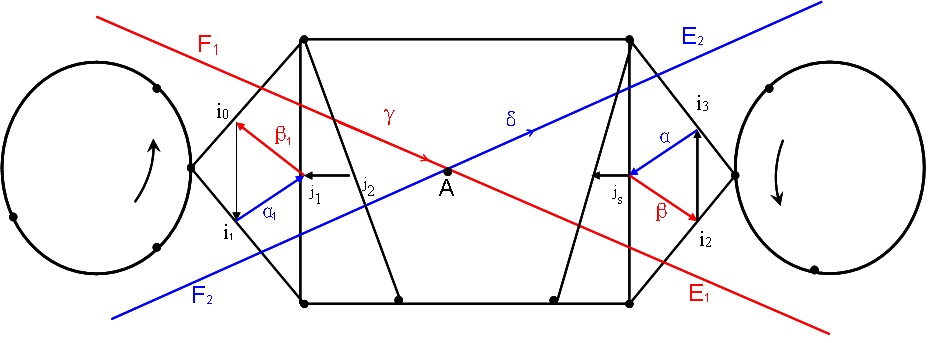}
\end{center}
In the above situation, we are going to find an admissible pair in $\ff'(\gz)\times \fs'(_s\da_e)$ induced by $A.$
By the construction of the triangulation $\ggz'$, we have $\gz, \da\not\in\ggz',$ and $s\geq 1$ which means the intersection
induces a common subword
$$w=j_1\frac{~~~~~~}{}j_2\frac{~~~~~~}{}\cdots\frac{~~~~~~}{} j_s$$
of $w(\ggz',\gz)$ and $w(\ggz',\da).$
Assume $w(\ggz',\gz)=E_1wF_1$ and $w(\ggz',\da)=E_2wF_2$, then $l(E_1)\neq 0\neq l(E_2)$ and $l(F_1)\neq 0\neq l(F_2)$ by
 the construction of $\ggz'$.

Note that $l(E_2)\neq 0 \neq l(F_2)$ implies that $i_3$ and $i_1$
are two arcs which induce two arrows $\az$ and $\az_1$ such that we
can write $w(\ggz',\da)=E_2wF_2=E'_2\az^{-1}w\az_1F'_2$. Therefore
$w(\ggz',{_s\da}_e)=E^{''}wF^{''}_2$ with $l(E^{''}_2)\geq 0$,
$l(F^{''}_2)\geq 0$. And $l(E^{''}_2)=0$ (resp. $l(F^{''}_2)=0$) if
and only if $E'_2=V(\az)^{-1}$ (resp. $F'_2=U(\az_1)$). Therefore
$(E^{''}_2,w,F_2'')\in \fs'(_s\da_e)$. Hence the intersection $A$
induces an admissible pair
    $$(E_1,w,F_1)\times(E^{''}_2,w,F^{''}_2)\in\ff'(\gz)\times \fs'(_s\da_e).$$

\end{proof}

\begin{rem}If $\emph{\Int}(\gz,\da)\neq 0$ with $_s\da_e\in\ggz'$, then ${_s\gz}_e\not\in\ggz'$ and
their intersections yield admissible pairs in
$\ff'(\da)\times\fs'(_s\gz_e).$
  \end{rem}
The following theorem is the main result in this section.

\begin{thm}\label{thm1}
Let $\gz, \da$ be curves (which are not necessarily different) in $(S,M),$ then
$$\emph{\mbox{dim}}_k \emph{\Ext}^1_{\C_{(S,M)}}(X_\gz,X_\da)=\emph{\Int}(\gz,\da).$$
\end{thm}
\begin{proof}
Let $\C'=\C_{(Q_{\ggz'},W_{\ggz'})}$, it suffices to prove
that $$\mbox{dim}_k \Ext^1_{\C'}(X_\gz,X_\da)=\mbox{Int}(\gz,\da).$$

We know that $$\Ext^1_{\C'}(X_\gz,X_\da)=\Hom_{\C'}(X_\gz,
X_\da[1])\cong (\add T_{\ggz'})(X_\gz,X_\da[1])\oplus
\Hom_{\C'/T_{\ggz '}}(X_\gz,X_\da[1])$$ where $(\add
T_{\ggz'})(X_\gz,X_\da[1])$ is the subspace of
$\Hom_{\C'}(X_\gz,X_\da[1])$ consisting of morphisms factoring
through an object in $\add T_{\ggz'}$, and $T_{\ggz'}$ is the
cluster tilting object corresponding to $\ggz'$ in $\C'$.

It follows from Lemma 3.3 in \cite{P08} that $$(\add
T_{\ggz'})(X_\gz,X_\da[1])\cong
D\Hom_{\C'/T_{\ggz'}}(X_\da,X_\gz[1]).$$

Therefore $\Ext^1_{\C'}(X_\gz,X_\da)$ can be decomposed as
$k$-vector space as follows
$$\Hom_{J(Q_{\ggz'},W_{\ggz'})}(M({\ggz'},\gz),M(\ggz',{_s\da}_e)) \oplus D\Hom_{J(Q_{\ggz'},W_{\ggz'})}(M(\ggz',\da),M(\ggz',{_s\gz}_e)).$$

Note that the definition of $w(\ggz', \gz)$ and $w(\ggz',\da)$ guarantees that
different intersections of $\gz$ and $\da$ yield different admissible pairs in $\ff'(\gz)\times \fs'(_s\da_e)$ or $\ff'(\da)\times\fs'(_s\gz_e).$

It follows from Theorem \ref{C-B} that
$$\dim_k\Hom_{J(Q_{\ggz'},W_{\ggz'})}(M({\ggz'},\gz),M(\ggz',{_s\da}_e))=\mid \ff'(\gz)\times \fs'(_s\da_e)\mid
$$
and
$$\dim_k\Hom_{J(Q_{\ggz'},W_{\ggz'})}(M(\ggz',\da),M(\ggz',{_s\gz}_e))=\mid\ff'(\da)\times\fs'(_s\gz_e)\mid.
$$
By Lemma \ref{admissible pair}, it suffices to show that each
admissible pair in $\ff'(\gz)\times \fs'(_s\da_e)$ or
$\ff'(\da)\times \fs'(_s\gz_e)$ can be induced by an intersection of
$\gz$ and $\da.$ Without loss of generality, we take an admissible
pair $(E_1,w,F_1)\times (E_2,w,F_2)\in\ff'(\gz)\times
\fs'(_s\da_e).$ By the orientation of the surface, we can have the
following picture in mind.
\begin{center}
\includegraphics[height=2in]{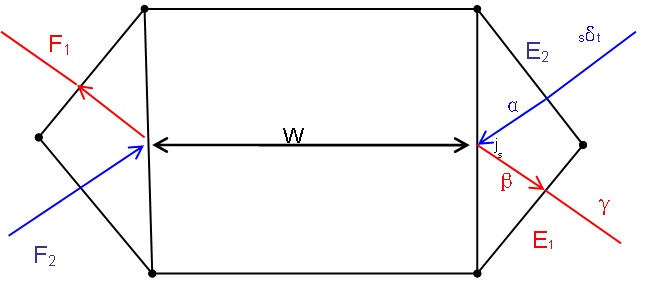}
\end{center}
It is easy to see that if $E_1=E_2$ or $E_2=F_2$,
then $t(\gz)=t(_s\da_e)$ or $s(\gz)=s(_s\da_e)$ which is impossible, since $e(\gz), s(\gz)\in M$ but $e(_s\da_e),s(_s\da_e)\in M'\backslash M$. So
$E_1\neq E_2$ and $F_1\neq F_2$, then the admissible pair is induced by an intersection of $\gz$ and $_s\da_e$ which can also be viewed
as an intersection of $\gz$ and $\da$ by definition of elementary pivot move and the structure of $\ggz'.$
This completes the proof.
\end{proof}

\begin{rem}let $\gz$ be a curve with
self-intersection in $(S,M)$ as follows. Then
$\emph{\Int}(\gz,\gz)=2,$
\begin{center}
\includegraphics[height=2in]{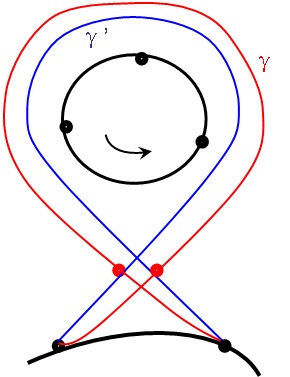}
\end{center}
where $\gz'$ is a curve in homotopy class of $\gz$ such that the
intersection number of  $\gz$ and $\gz'$ is minimal. Note that
$\emph{\Int}(\gz,\gz)$ may be different from the ordinary notion of
self-intersection number of $\gz$.
\end{rem}

\section{Classification of cotorsion pairs in $\C_{(S,M)}$}

Let $(S,M)$ be a marked surface without punctures and $\C_{(S,M)}$
be the corresponding cluster category. We show in this section that
there is no non-trivial $t-$structures in $\C_{(S,M)}$ when $S$ is
connected. Moreover, we give a classification of cotorsion pairs
with a fixed core in $\C_{(S,M)}$. Recall that we denote by $\X_V$
the subcategory of $\C_{(S,M)}$ corresponding to a collection $V$ of
curves and valued closed curves in $(S,M)$.

\begin{lem}\label{move}
Let $\X_V$ be a subcategory of $\C_{(S,M)}$ closed under extensions
and the shift functor $[1]$. Then for any curve $\gamma\in V$ and
any positive integer $m$, ${_{s^m}\gamma}_{e^m}$, $_{s^m}\gamma$,
$\gamma_{e^m}$ are also in $V$.
\end{lem}

\begin{proof}
Since $\X_V$ is closed under the shift functor $[1]$, then for any
curve $\gamma\in V$, $X_{_s\gamma_e}=X_\gamma[1]$ is also in $\X$.
Then the objects $X_{_s\gamma}$ and $X_{\gamma_e}$ are also in $\X$
since $\X$ is closed under extensions, and $X_{{}_s\gamma}\oplus
X_{\gamma_e}$ is the middle item of the AR-triangle ending at
$X_\gamma$. By induction on $m$ , we have the assertion.
\end{proof}

For any curve $\gamma$ in $(S,M)$, we denote by $\B(\gamma)$ the set
of boundaries where the endpoints of $\gz$ lie on. Set
$B(b,\lambda):=\emptyset$ for a closed curve $b$. For a collection
$V$ of curves and valued closed curves, let
$$\B(V)=\displaystyle\bigcup_{\da\in V}\B(\da).$$

\begin{lem}\label{boundary}
Let $(\X_V,\X_W)$ be a t-structure in $\C_{(S,M)}$. Then $\B(V)\bigcap\B(W)=\emptyset$.
\end{lem}

\begin{proof}
Assume that $\B(V)\bigcap\B(W)\neq\emptyset$. Then there are two
curves $\gamma\in V$, $\da\in W$ and a boundary $B$ such that
$B\in\B(\gamma)\bigcap\B(\da)$. Let $a_0=e(\gz), a_1, \ldots,
a_{n-1}$ be all the marked points on $B$ with anticlockwise order.
Then the definition of pivot moves implies
$e(\gamma_{e^i})=a_i=a_{i+n}$.
Then any curve with endpoint $a_i$
crosses $\gamma_{e^{i-1}}$ or $\gamma_{e^{i+n+1}}$ for $i\geq1$ (see
the following figure for the case $i=1$).
\begin{center}
\includegraphics[height=2in]{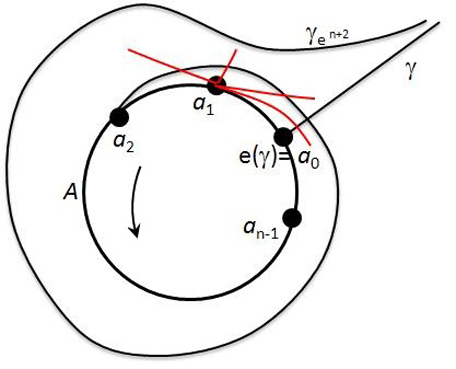}
\end{center}
By Lemma \ref{move}, both $\gamma_{e^{i-1}}$ and
$\gamma_{e^{i+n+1}}$ are in $V$. Hence, there exists $m\geq 0$ such
that $\Int(\gz_{e^m}, \da )\neq 0$, where $\gz_{e^m}\in V$. Theorem
\ref{thm1} implies $\Ext^1_{\C{(S,M)}}(X_{\gz_{e^m}}, X_{\da})\neq0$
which contradicts to the definition of $t-$structure.
\end{proof}

The following theorem is the first main result in this section.

\begin{thm}\label{thm2}
If $S$ is connected, then the $t-$structures in $\C_{(S,M)}$ are $(\C_{(S,M)},0)$ or $(0,\C_{(S,M)})$.
\end{thm}

\begin{proof}

Suppose that $(\X_V,\X_W)$ is a $t-$structure in $\C_{(S,M)}$. We first prove $$\Ext^1(\X_W[1],\X_V)=0$$ in the following three cases:
\begin{itemize}
  \item[$(1)$]Let $X_\gz$ and $X_\da$ be two string objects in $\C_{(S,M)}$ with $\gamma\in V$ and $\delta\in W.$ Since $\Ext^1(X_{\gamma},X_{\delta})=0$ by the definition of t-structure, $\gamma$ does not cross $\delta$ by Theorem \ref{thm1}. Then $\gamma$ does not cross $_s\delta_e$ neither, since $\B(\gamma)\bigcap\B(\delta)=\emptyset$ by Lemma \ref{boundary} and $_s\delta_e$ is obtained from $\delta$ by moving the endpoints of $\delta$ along the boundary. Then $\Ext^1(X_{\delta}[1],X_{\gamma})=0$ by Theorem \ref{thm1}.

   \item[$(2)$]For any band object $X$ in $\X_V$ and any object $Y$ in $\X_W$, $\Ext^1(X,Y)=0$ implies that
$$\Ext^1(Y[1],X)\cong D\Ext^1(X,Y[1])\cong D\Ext^1(X[-1],Y)\cong D\Ext^1(X,Y)=0$$
since the band object is invariant under $[-1]$.
 \item[$(3)$] For any object $X$ in $\X_V$ and any band object $Y$ in $\X_W$, $\Ext^1(X,Y)=0$ implies that
$$\Ext^1(Y[1],X)\cong D\Ext^1(X,Y[1])\cong D\Ext^1(X,Y)=0$$
since the band object is invariant under $[1]$.
\end{itemize}

Therefore $\Ext^1(\X_W[1],\X_V)=0$. Then any indecomposable object in $\C_{(S,M)}$ is either in $\X_V$ or in $\X_W[1]$.
Hence any curve is either in $V$ or in $W$. Since $\B(V)\bigcap\B(W)=\emptyset$ and $S$ is connected, we have that either $V$ or $W$ contains all curves in $(S,M)$. Without loss of generality, we assume that $V$ contains all curves in $(S,M)$. In particular, all arcs in $\Gamma$ are in $V$. So there is a cluster tilting object $T=\bigoplus_{\gamma\in\Gamma}X_\gamma$ in $\X_V$. Then $\C_{(S,M)}=\add T\ast\add T[1]\subset\X_V$. This completes the proof.
\end{proof}

If the surface $(S,M)$ has $m$ connected components $(S_j,M_j)$, then the corresponding cluster category $\C_{(S,M)}$ is equivalent to the direct sum of cluster categories $\C_{(S_j,M_j)}$. So we have the following corollary.

\begin{cor}\label{cor1}
Let $(S,M)$ be a marked surface with $m$ connected components $(S_j,M_j)$, $1\leq j\leq m$. Then there are $2^m$ $t-$structures in $\C_{(S,M)}$, and they are of the form $$(\bigoplus_{j\in J}\C_{(S_j,M_j)}, \bigoplus_{j\notin J}\C_{(S_j,M_j)})$$
where $J\subseteq\{1,\ldots,m\}$.
\end{cor}

Let $\X_I$ be a rigid subcategory of $\C_{(S,M)}$. For a subcategory $\X_V\supset \X_I$, we denote by $\overline{\X_V}=\X_V/\X_I$, the subcategory of the triangulated category ${}^{\bot}\X_I[1]/\X_I$. Recall Theorem 3.5 of \cite{ZZ} in our setup: a pair $(\X_{V^1},\X_{V^2})$ of subcategories is a cotorsion pair with core $\X_I$ in $\C_{(S,M)}$ if and only if $\X_I\subset\X_{V^i}\subset{}^{\bot}\X_I[1] \mbox{~for~} i=1,2$,  and $(\overline{\X_{V^1}},\overline{\X_{V^2}})$ is a $t-$structure in ${}^{\bot}\X_I[1]/\X_I.$ This theorem allows us to give a classification of cotorsion pairs  with a fixed core in $\C_{(S,M)}$. We recall that $\V(S,M)$ is the collection of all curves and valued closed curves in
$(S,M)$. By $\V_{I}(S,M)$, we mean the collection of all curves and closed curves in $(S,M)$ which do not belong to $I$ such that
they do not cross any arcs in $I$.
\begin{thm}\label{thm3}
Let $\X_I$ be a rigid subcategory of $\C_{(S,M)}$ such that $(S,M)/I$ has $m$ connected components $(S^I_j,M^I_j)$, $1\leq j\leq m$. Then there is a bijection from the power set of $\{1,2,\ldots,m\}$ to the set of cotorsion pairs in $\C_{(S,M)}$ with core $\X_I$:
$$J\mapsto (\X_{I\cup\bigcup_{j\in J}\V{(S^I_j, M^I_j)}}, \X_{I\cup\bigcup_{j\in J^c}\V{(S^I_j,M^I_j)}})=:(\X(J), \X(J^c))$$
where $J\subseteq\{1,2,\ldots,m\}$ and $J^c=\{1,2,\ldots,m\}\setminus J$. In particular, there are exactly $2^m$ cotorsion pairs with core $\X_I$ in $\C_{(S,M)}$.
\end{thm}

\begin{proof}

Note that $\V_{I}(S,M)=\bigcup_{j=1}^m\V(S^I_j,M^I_j)$ and ${}^{\bot}\X_I[1]=\X_{I\cup\V_{I}(S,M)}$ by definition. Therefore
$I\subset V \subset I\cup\bigcup_{j=1}^m\V(S^I_j,M^I_j)$ if and only if $\X_I\subset\X_{V}\subset{}^{\bot}\X_I[1]$. On the other hand, the equivalent functor in Proposition \ref{propMP} implies a bijection between $t-$structures
$(\overline{\X_{V^1}},\overline{\X_{V^2}})$ in ${}^{\bot}\X_I[1]/\X_I$ and $t-$structures $(\X_{V^1\setminus I},\X_{V^2\setminus I})$ in $\C_{(S,M)/I}.$
Combine this together with Corollary \ref{cor1} and Theorem 3.5 in \cite{ZZ11}, we get the proof.

\end{proof}

\begin{exm}\label{ex1}
Let $(S,M)$ be obtained from a sphere by removing 3 disks, with two marked points on each boundary component, and $I=\{\gamma_1,\gamma_2,\gamma_3,\gamma_4\}$ be a collection of arcs in $(S,M)$. See the following figure.

\begin{center}
\includegraphics[height=2in]{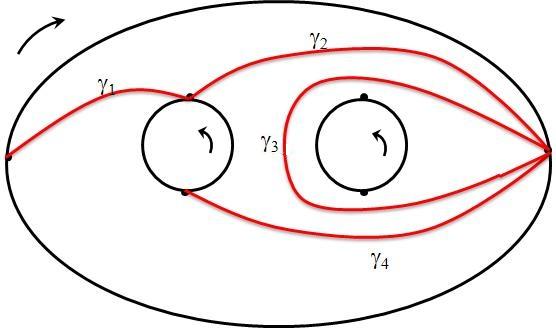}
\end{center}

Therefore $(S,M)/I$ have three connected components:

\begin{center}
\includegraphics[height=2in]{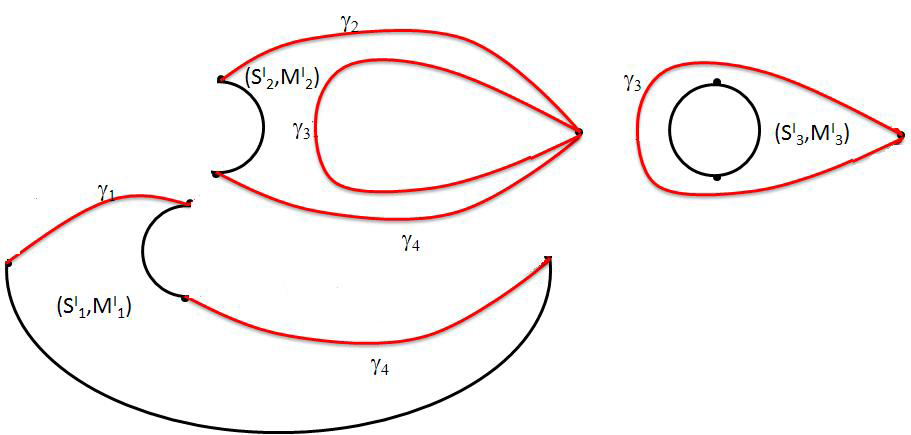}
\end{center}

The above theorem implies that there are eight cotorsion pairs with core $\X_I$ in $\C_{(S,M)}$.
\end{exm}

\medskip

Since the two parts of a cotorsion pair have the same form, we have the following corollary.

\begin{cor}
 If $(\X,\Y)$ is a cotorsion pair in $\C_{(S,M)}$, then so is $(\Y,\X)$, and they have the same core.
\end{cor}

Recall that a co-$t-$structure $(\X,\Y)$ is a cotorsion pair with $\X[-1]\subset\X$ and $\Y[1]\subset\Y$.

\begin{cor}
There is no nontrivial co-$t-$structures in the cluster category of a connected marked surface.
\end{cor}

\begin{proof}
Let $(\X,\Y)$ be a co-$t-$structure. Then $(\X,\Y)$ is a cotorsion pair which implies $(\Y,\X)$ is also a cotorsion pair by the above corollary. Therefore $(\Y,\X)$ is a $t-$structure. By Theorem \ref{thm2}, $\X=0$ or $\C$.
\end{proof}

\section{A geometric model of cotorsion pairs and their mutations in $\C_{(S,M)}$}

Let $(S,M)$ be a marked surface without punctures and $\C_{(S,M)}$ be the corresponding cluster category. We give in this section a geometric description of mutations of cotorsion pairs in $\C_{(S,M)}$.

By the structure of cotorsion pairs given in Theorem \ref{thm3}, we provide a geometric model of cotorsion pairs.

\begin{defn}
Let $I$ be a collection of pairwise compatible arcs such that $(S,M)/I$ has $m$ components $(S^I_1,M^I_1),\ldots,
(S^I_m, M^I_m)$. For each $J\subseteq\{1,\ldots,m \}$, an $I_J-$painting of $(S,M)$ is obtained from $(S,M)$ by filling black in $(S^I_i,M^I_i)$ for each $i\in J$ and leaving other components white.
\end{defn}

\begin{exm}
Let $(S,M)$ and $I$ be the same in Example \ref{ex1}. Then there are 8 $I_J-$paintings of $(S,M)$, see the following figure.

\begin{center}
\includegraphics[height=1.7in]{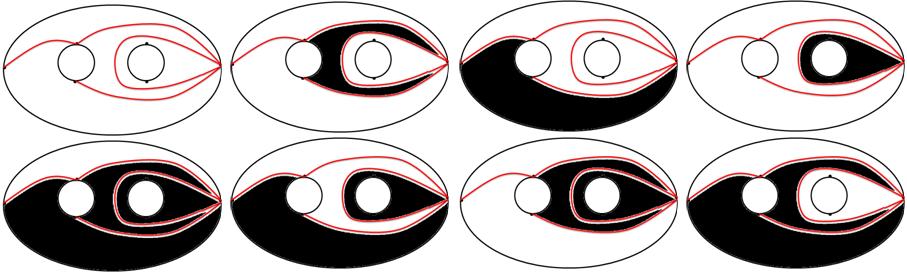}
\end{center}
\end{exm}

Note that Theorem \ref{thm3} implies that there is a bijection between the set of $I_J-$paintings of $(S,M)$ and cotorsion pairs with core $\X_I$ in $\C_{(S,M)}$. The bijection sends the black components in an $I_J-$painting of $(S,M)$ to the left part $\X(J)$ of a cotorsion pair $(\X(J),\X(J^c))$, and the white components go to $\X(J^c)$.

\begin{rem}
When $S$ is a disk, the definition of paintings of $(S,M)$ is the same as the definition of
Ptolemy diagrams in \cite{HJR11} (see Theorem A(ii) and Remark 2.6 in \cite{HJR11}).
\end{rem}

Let $(\X_{V^1},\X_{V^2})$ be a cotorsion pair in $\C_{(S,M)}$ with
core $\X_I$ and $\X_D$ be a subcategory of $\C_{(S,M)}$ with
$D\subset I$. Recall that the $\X_D-$mutation $(\mu
^{-1}((\X_{V^1},\X_{V^2});\X_D)$ of $(\X_{V^1},\X_{V^2})$,
introduced in \cite{ZZ}, is defined as a new cotorsion pair
$(\X_D\ast \X_{V^1} [1])\cap {}^{\bot}(\X_D [1]), (\X_D\ast \X_{V^2}
[1])\cap {}^{\bot}(\X_D [1]))$. In particular, $0-$mutation of
$(\X_{V^1},\X_{V^2})$ is just the cotorsion pair
$(\X_{V^1}[1],\X_{V^2}[1])$. Let $$\mu
^{-1}(\X_{V^i};\X_D)=(\X_D\ast \X_{V^i} [1])\cap {}^{\bot}(\X_D
[1]),$$for $i=1$, $2$, the pairs
$(\X_{V^i},\mu^{-1}(\X_{V^i};\X_D))$ are called $\X_D-$mutation
pairs in $\C_{(S,M)}$ (see Definition 2.5 in \cite{IY08}).

\medskip

In the following, we define the rotation of $I_J-$paintings.

\begin{defn} Let $I$ be a collection of pairwise compatible arcs in $(S,M)$ such
that $(S,M)/I$ has $m$ components $(S^I_1,M^I_1),\ldots,
(S^I_m, M^I_m)$, and $D$ be a subcollection of
$I$. Note that each component $(S^D_j,M^D_j)$ of $(S,M)/D$ inherits
an orientation from the orientation of $S$, and contains several
components $(S^I_{j_i},M^I_{j_i})$ of $(S,M)/I$ where
$j_i\in\{1,2,\ldots,m\}$. The $D-$rotation of
$(S^I_{j_i},M^I_{j_i})$ is induced by applying elementary pivot moves on both endpoints of each
curve in $(S^I_{j_i},M^I_{j_i})$, but along the boundary of $(S^D_j,M^D_j)$.
\end{defn}

See the following picture:
\begin{center}
\includegraphics[height=1.5in]{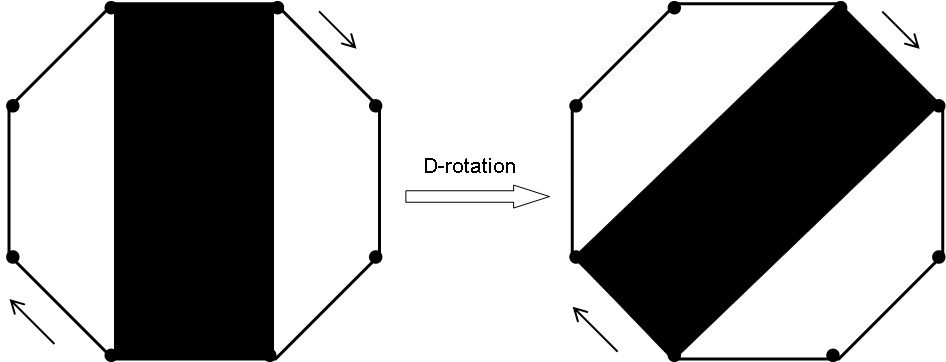}
\end{center}
where on the left side, the black quadrangle is $(S^I_{j_i},M^I_{j_i})$ in $(S^D_j,M^D_j)$ denoted by the octagon, then the $D-$rotation of $(S^I_{j_i},M^I_{j_i})$ is the black quadrangle on the right side of the above figure.

\begin{defn} The $D-$rotation of an $I_J-$painting of $(S,M)$ is defined to be a new painting of $(S,M)$ induced by the $D-$rotation of all components $(S^I_{i},M^I_{i})$ of $(S,M)/I$.
\end{defn}

\begin{exm}
We take one $I_J-$painting in Example \ref{ex1} where $J=\{1,3\}$. Let $D=\{\gamma_2,\gamma_4\}$. Then the $D-$rotation of this $I_J$-painting can be described as follows.
\begin{center}
\includegraphics[height=2.8in]{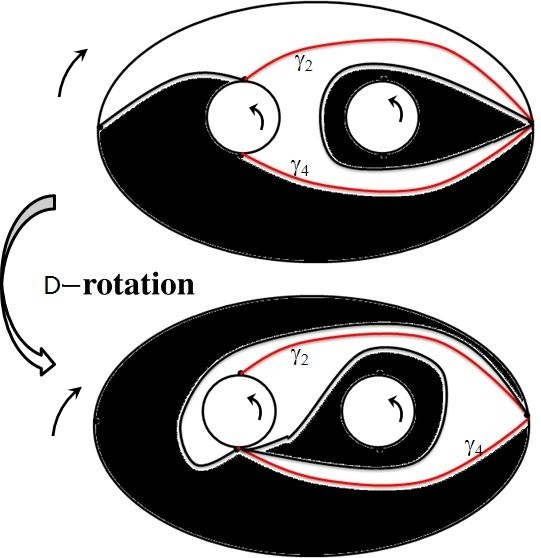}
\end{center}
\end{exm}

Finally, we show the rotations of paintings of $(S,M)$ provide a geometric model of mutations of cotorsion pairs in $\C_{(S,M)}$.

\begin{thm}\label{rotation}
Let $\X_I$ be a rigid subcategory of $\C_{(S,M)}$ such that
$(S,M)/I$ has $m$ connected components $(S^I_j,M^I_j)$, $1\leq j\leq
m$, and $\X_D$ be a subcategory of $\C_{(S,M)}$ with $D\subset I$.
Then the $\X_D-$mutations of cotorsion pairs $(\X(J),\X(J^c))$ are compatible with
the $D-$rotations of $I_J-$paintings of $(S,M)$, under the bijection between
cotorsion pairs in $\C_{(S,M)}$ and paintings of $(S,M)$.
\end{thm}
\begin{proof}
Assume $(\X(J),\X(J^c))$ is a cotorsion pair in $\C_{(S,M)}$ with core
$\X_I,$ where $J\subseteq \{1,2,\ldots,m\}.$ Then Proposition 4.9 in \cite{IY08} implies
that for any $\X_D-$mutation pair $(\X,\mu^{-1}(\X;\X_D))$ in
$\C_{(S,M)}$, $(\overline{\X},\overline{\mu^{-1}(\X;\X_D)})$ forms a
$0-$mutation pair in $\C_{\X_D}={}^\bot\X_D[1]/\X_D$. Hence,
$$\overline{\mu^{-1}(\X(J);\X_D)}=\overline{\X(J)}\langle1\rangle,~~\overline{\mu^{-1}(\X(J^c);\X_D)}=\overline{\X(J^c)}\langle1\rangle$$
 where $\langle1\rangle$ is the shift functor in $\C_{\X_D}$. It follows from the equivalence between $\C_{\X_D}$ and $\C_{(S,M)/D}$ (compare Proposition 2.4) that the funtor $\langle1\rangle$ corresponds to the pivot elementary moves on both endpoints in $(S,M)/D$. Therefore the $\X_D-$mutation of $(\X(J),\X(J^c))$ is compatible with the $D-$rotation of the corresponding $I_J-$painting of $(S,M).$
\end{proof}

\begin{rem}  The above theorem is a generalization of Theorem 4.4 in \cite{ZZ11}.
When $S$ is a disk, the rotation of paintings is the mutation of Ptolemy diagrams in \cite{ZZ11}.
\end{rem}

{\bf Acknowledgments}

\medskip

The authors would like to thank Peter J{\o}rgensen for his comments and suggestions, they are very
grateful to him for pointing out an error in a previous version of Definition \ref{def}(2).

\end{document}